\documentclass[a4paper,10pt,leqno]{amsart}
\usepackage{amsmath,amsfonts,amsthm,amssymb,dsfont}
\usepackage[alphabetic]{amsrefs}
\usepackage[OT4]{fontenc}
\usepackage{enumerate}
\usepackage{mathrsfs}
\usepackage{mathtools}
\usepackage{enumerate}
\usepackage{float}
\usepackage[dvipsnames]{xcolor}

\newtheorem{thm}{Theorem}[section]
\newtheorem{lem}[thm]{Lemma}
\newtheorem{prop}[thm]{Proposition}
\newtheorem{cor}[thm]{Corollary}
\newtheorem{theoremalph}{Theorem}
\newtheorem{coralph}[theoremalph]{Corollary}

\theoremstyle{definition}
\newtheorem{rem}[thm]{Remark}
\newtheorem{defin}[thm]{Definition}

\newtheorem{conv}[thm]{Convention}
\newtheorem{ex}[thm]{Example}

\newcommand{\eps}{\varepsilon}
\newcommand{\Z}{\mathbb{Z}}
\newcommand{\R}{\mathbb{R}}
\newcommand{\an}{\measuredangle}
\newcommand{\Id}{\mathrm{Id}}
\newcommand{\St}{\mathrm{Stab}}

\makeatletter
\def\paragraph{\@startsection{paragraph}{4}%
  \z@\z@{-\fontdimen2\font}%
  {\normalfont\bfseries}}
\makeatother

\begin{document}

\title[Acylindrical actions for two-dimensional Artin groups]{Acylindrical actions for two-dimensional Artin groups of hyperbolic type}

\author[A.~Martin]{Alexandre Martin$^{\dag*}$}
           \address{Department of Mathematics and the Maxwell Institute for Mathematical Sciences,
           Heriot-Watt University,
           Riccarton,
           EH14 4AS Edinburgh, United Kingdom}
           \email{alexandre.martin@hw.ac.uk}
           \thanks{$\dag$ Partially supported by EPSRC New Investigator Award EP/S010963/1.}

\author[P.~Przytycki]{Piotr Przytycki$^{\ddagger*}$}
\address{
Department of Mathematics and Statistics,
McGill University,
Burnside Hall,
805 Sherbrooke Street West,
Montreal, QC,
H3A 0B9, Canada}

\email{piotr.przytycki@mcgill.ca}
\thanks{$\ddagger$ Partially supported by NSERC, FRQNT, and National Science Centre, Poland UMO-2018/30/M/ST1/00668.}
\thanks{$*$ This work was partially supported by
the grant 346300 for IMPAN from the Simons Foundation and the matching
2015--2019 Polish MNiSW fund}

\maketitle

\begin{abstract}
For a two-dimensional Artin group $A$ whose associated Coxeter group is hyperbolic, we prove that the action of $A$ on the hyperbolic space obtained by coning off certain subcomplexes of its modified Deligne complex is acylindrical. Moreover, if for each $s\in S$ there is $t\in S$ with $m_{st}< \infty$, then this action is universal. As a consequence, for $|S|\geq 3$, if $A$ is irreducible, then it is acylindrically hyperbolic. We also obtain the Tits alternative for~$A$, and we classify the subgroups of $A$ that virtually split as a direct product. A key ingredient in our approach is a simple criterion to show the acylindricity of an action on a two-dimensional $\mathrm{CAT}(-1)$ complex.
\end{abstract}

\section{Introduction}

Artin groups form a class of intensively studied groups generalising braid groups and closely related to Coxeter groups. Let us recall their definition. Let $S$ be a finite set and for all $s\neq t\in S$ let $m_{st}=m_{ts}\in \{2,3,\ldots, \infty\}$. We encode this data in the \emph{defining graph} with vertex set $S$ and edges between $s$ and $t$ with label $m_{st}$ whenever $m_{st}<\infty$. The associated \emph{Artin group} $A_S$ is given by generators and relations:
$$ A_S = \langle S \ | \ \underbrace{sts\cdots}_{m_{st}}=\underbrace{tst\cdots}_{m_{st}}\rangle.$$ The \emph{associated Coxeter group} $W_S$ is obtained by adding the relations $s^2=1$ for every $s\in S$.

\smallskip

Coxeter groups are well understood in many respects. In particular, they are known to be $\mathrm{CAT}(0)$ \cite{Mou}. Artin groups on the other hand form a class of groups with a more mysterious structure and geometry, and many problems remain open in general: whether they are torsion-free, linear, or whether they satisfy the celebrated $K(\pi, 1)$ conjecture. (See \cite{CharneyProblems} for a survey of many open problems on Artin groups, as well as partial results.) While little is known about general Artin groups, they are expected to be as well-behaved as Coxeter groups. In particular, braid groups are conjectured to be  $\mathrm{CAT}(0)$, which was verified in low dimensions \cite{4BraidCAT0, 6BraidCAT0}. Recently, Artin groups \textit{of large-type} (i.e.\ such that $m_{st}\geq 3$ for every $s\neq t \in S$) were shown to be systolic \cite{SystolicArtin}, a simplicial analogue of $\mathrm{CAT}(0)$. Artin groups of \emph{XXL type} (that is, with all $m_{st}\geq 5$), are CAT(0) by \cite{H2}.

\smallskip

A unifying theme for many groups in geometric group theory has been to find interesting actions on hyperbolic spaces, and in particular \textit{acylindrical} actions on hyperbolic spaces. Recall that an isometric action of a group $G$ on a metric space~$X$ is acylindrical \cite{BowditchTightGeodesics} if for every $r \geq 0$, there exist constants $L, N\geq 0$ such that for every $x, y \in X$ at distance at least $L$, there are at most $N$ elements $g \in G$ such that $d(x, gx) \leq r$, $d(y, gy)\leq r$. The prime example of this phenomenon is the mapping class group of a closed hyperbolic surface acting acylindrically on its curve complex \cite{MasurMinsky, BowditchTightGeodesics}. Since then, many groups have been shown to admit acylindrical actions on hyperbolic spaces \cite{OsinAcyl}, including large classes of Artin groups and related groups  \cite{ChatterjiMartin, CalvezWiest, CharneyAcylArtin, H2}.

We say that an action is \emph{elliptic} if it has bounded orbits. A group that is not virtually cyclic is \emph{acylindrically hyperbolic} if it has has an acylindrical action on a hyperbolic space that is not elliptic. (For equivalent definitions, see \cite[Thm~1.2]{OsinAcyl}.)

While a group may have many acylindrical actions on hyperbolic spaces, there has been a lot of interest recently in understanding which groups, like mapping class groups, admit a `universal' acylindrical action on a hyperbolic space. Let $G$ be a group acting on a hyperbolic space $X$. Recall that an element $g \in G$ is \textit{loxodromic} for this action if for some (hence any) point $x \in X$, its orbit map $\Z\ni n\to g^nx\in X$ is a quasi-isometric embedding. An acylindrical action of a group $G$ on a hyperbolic metric space $X$ is called \textit{universal} \cite{OsinAcyl} if the elements of $G$ that are loxodromic for this action are exactly the elements of $G$ that are loxodromic for \textit{some} acylindrical action of $G$ on a hyperbolic space (such elements are called \textit{generalised loxodromic}). Universal actions often offer much deeper insight into the structure of the underlying group, and are known to exist for instance for right-angled Artin groups, and more generally for hierarchically hyperbolic groups \cite{HHUniversal}. The goal of this article is to show the existence of such a universal acylindrical action for a large class of Artin groups, and to use the dynamics of the action to understand the structure of certain of its subgroups.

\medskip

\paragraph{Statement of results.} An Artin group is \textit{two-dimensional} (see \cite{CD}) if for every $s, t, r \in S$, we have
$$\frac{1}{m_{st}} + \frac{1}{m_{tr}}+\frac{1}{m_{sr}} \leq 1.$$ This class contains in particular all large-type Artin groups. We say that an Artin group is of \textit{hyperbolic type} if the associated Coxeter group is hyperbolic. For two-dimensional Artin groups, by \cite{Mou} this is equivalent to requiring that for any $s, t, r \in S$, we have
$$\frac{1}{m_{st}} + \frac{1}{m_{tr}}+\frac{1}{m_{sr}} < 1.$$

An important complex associated to an Artin group is its \textit{modified Deligne complex} (see Definition \ref{def:Deligne}) introduced by Charney and Davis \cite{CD}, and generalising a construction of Deligne for Artin groups of spherical type \cite{Del}. While the action of an Artin group on its modified Deligne complex is almost never acylindrical, we are able to construct an acylindrical action for a two-dimensional Artin group of hyperbolic type by coning off an appropriate family of subcomplexes of its modified Deligne complex called \textit{standard trees} (see Definition \ref{def:standard_tree}). Our main result is the following:

\begin{theoremalph}
\label{thm:acylindrically}
The action of a two-dimensional Artin group $A_S$ of hyperbolic type on its coned off Deligne complex is acylindrical.
Moreover, if for each $s\in S$ there is $t\in S$ with $m_{st}< \infty$, then this action is universal.

In particular, for $|S|\geq 3$, if $A_S$ is irreducible, then it is acylindrically hyperbolic.
\end{theoremalph}

Here $A_S$ is \emph{irreducible} if there is no nontrivial partition $S=T\sqcup T'$ with $m_{tt'}=2$ for all $t\in T,t'\in T'$.

Note that a notion stronger than universal acylindrical action, namely \textit{largest} acylindrical action, was recently introduced in \cite{HypStructures}. Right-angled Artin groups and more generally hierarchically hyperbolic groups are known to have such a largest action \cite{HHUniversal}. Although we do not address this question here, it would be interesting to know whether the action of $A_S$ on its coned off Deligne complex is a largest acylindrical action for $A_S$.

Theorem A is the first example of a universal acylindrical action for Artin groups that are not right-angled. The dynamics of such an action can be used to understand the structure of certain subgroups of these Artin groups, as we now explain.

\medskip

\paragraph{Tits alternative.} A group satisfies the \textit{Tits alternative} if every finitely generated subgroup is either virtually soluble or contains a non-abelian free subgroup. This dichotomy has been shown for many groups of geometric interest, such as linear groups \cite{TitsAlternative}, mapping class groups of hyperbolic surfaces \cite{IvanovAutomorphisms, McCarthyTitsAlternative}, outer automorphism groups of free groups \cite{BestvinaFeighnHandelTitsI,BFH2}, groups of birational transformations of surfaces \cite{CantatTits}, etc. A general heuristic is that a group that is non-positively curved in a very broad sense should satisfy the Tits alternative. In particular, it is conjectured that all $\mathrm{CAT}(0)$ groups satisfy the Tits alternative. Several classes of $\mathrm{CAT}(0)$ groups have been shown to satisfy this alternative, in particular cocompactly cubulated groups \cite{SageevWiseTits}, and groups acting geometrically on two-dimensional systolic complexes or buildings \cite{OP}, but the problem remains open in general.

Following the heuristic that Artin groups should be non-positively curved in an appropriate sense, it is natural to ask whether Artin groups satisfy the Tits alternative, as already noted in \cite{BestvinaArtin}. It should be mentioned that  Coxeter groups are linear, which implies that they do satisfy the Tits alternative. So far, the following classes of Artin groups have been shown to satisfy the Tits alternative:

\begin{itemize}
\item Artin groups that can be cocompactly cubulated. An important class of such Artin groups is the class of \textit{right-angled} Artin groups, namely Artin groups such that $m_{st} = 2$ or $\infty$ for every $s\neq t \in S$. Beyond them, a few classes of Artin groups have been shown to be cocompactly cubulated \cite{HJP, VirtuallyCubulatedArtin}, but the conjectural picture states that the class of cocompactly cubulated Artin groups is extremely constrained \cite{VirtuallyCubulatedArtin}.
\item Artin groups \textit{of finite type} (also known as \textit{spherical} Artin groups), i.e.\ Artin groups whose associated Coxeter group is finite, since they were shown to be linear \cite{Kr, SphericalArtinLinear, Dig}.
\item Artin groups acting geometrically on certain two-dimensional complexes, including large-type Artin groups  \cite[Thm~A.2]{OP} and Artin groups acting geometrically on a two-dimensional systolic complex \cite[Thm~A]{OP} provided by \cite{BM}.
\end{itemize}

We obtain a strengthening of the Tits alternative for two-dimensional Artin groups of hyperbolic type.

\begin{coralph}[Tits alternative]
\label{thm:Tits Alternative} Let $A_S$ be a two-dimensional Artin group of hyperbolic type. Then $A_S$ satisfies the Tits alternative. More precisely, every subgroup of $A_S$ that is not virtually cyclic either contains a nonabelian free group or is virtually~$\Z^2$.
\end{coralph}

In a forthcoming article \cite{MP2}, we prove the Tits alternative for
Artin groups of type~FC, using different techniques.

\medskip

\paragraph{Virtual abelian and virtual product subgroups.} We complete Corollary~\ref{thm:Tits Alternative}  by obtaining a classification of the virtually $\mathbb{Z}^2$ subgroups of two-dimensional Artin groups of hyperbolic type. Before stating our result, let us first recall some of the known $\mathbb{Z}^2$ of Artin groups (see Section~\ref{sec:preliminaries} for references).

A \textit{dihedral} Artin group, that is, an Artin group $A_{st}$ on two generators $s, t$ with $m_{st}< \infty$, contains a finite index subgroup isomorphic to $\mathbb{Z} \times F$, where $F$ is a free group of rank $\geq 1$. In particular, dihedral \emph{parabolic} subgroups (i.e.\ conjugates of $A_{st}$) of a given Artin group $A_S$ contain many $\mathbb{Z}^2$ subgroups.

Another source of $\mathbb{Z}^2$ subgroups are the central elements of dihedral parabolic subgroups. The centre of a dihedral Artin group on two generators $s, t$ with $m_{st}\geq 3$ is infinite cyclic, generated by an element $z_{st}$. In particular, for a general Artin group and $s \in S$, $s$ commutes with any element in the subgroup generated by $z_{st_1}, \ldots, z_{st_m}$, where $t_1, \ldots, t_m$ denote the neighbours of $s$ is the defining graph.

For two-dimensional Artin groups of hyperbolic type, our result states that these $\mathbb{Z}^2$ subgroups are close to being the only ones, and relies crucially on the dynamics of the action of these Artin groups on their coned off Deligne complex.

\begin{coralph}[Classification of virtually $\mathbb{Z}^2$ subgroups]
\label{thm:classification_Z2}
Let $A_S$ be a two-dimensional Artin group of hyperbolic type, and let $H$ be a subgroup that is virtually $\mathbb{Z}^2$. Up to conjugation, one of the following occurs:
\begin{itemize}
\item $H$ is contained in the stabiliser of a vertex of the modified Deligne complex (i.e.\ is contained in a dihedral parabolic subgroup), or
\item $H$ is contained in the stabiliser of a standard tree of the modified Deligne complex. In particular, $H$ contains a conjugate of a non-trivial power of some $s \in S$. We refer to Remark \ref{rem:description} for an explicit description of these subgroups.
\end{itemize}
\end{coralph}

By contrast, for more general two-dimensional Artin groups where $A_S$ contains a Euclidean parabolic subgroup, there exist `exotic' virtually $\mathbb{Z}^2$ subgroups coming from periodic flats of the modified Deligne complex, see \cite{MPabelian}.

\smallskip

In a similar direction, we also obtain a complete classification of the subgroups that decompose as a non-trivial product in a two-dimensional Artin group of hyperbolic type.

\begin{coralph}[Classification of virtual products]
\label{thm:classification_prod}
Let $A_S$ be a two-dimensional Artin group of hyperbolic type, and let $H$ be a  subgroup that virtually splits as a (non-trivial) direct product. Then $H$ is virtually of the form $\mathbb{Z}\times F$, where $F$ is a free group.
\end{coralph}

\paragraph{Strategy of the proof.} The key to all the theorems is to find a convenient hyperbolic space on which $A_S$ acts acylindrically.
The first space we study is the modified Deligne complex $\Phi$ associated to $A_S$, where it was shown that the Moussong metric on $\Phi$ is $\mathrm{CAT}(0)$ for two-dimensional Artin groups \cite{CD}. In the case of two-dimensional Artin groups of hyperbolic type, this metric can be modified to be $\mathrm{CAT}(-1)$. However, since dihedral Artin groups have non-trivial centres, the action on $\Phi$ is not acylindrical. More precisely, $\Phi$ contains unbounded trees with infinite pointwise stabilisers (see Definition \ref{def:standard_tree}). To circumvent this, we construct a new space $\Phi^*$ by coning off these trees, and we show that it is still possible to endow this new space $\Phi^*$ with an equivariant $\mathrm{CAT}(-1)$ metric. More crucially, removing these obvious obstructions to acylindricity turns out to be enough, as we prove that the action of $A_S$ on the coned off Deligne complex $\Phi^*$ is acylindrical and universal.

Theorem~\ref{thm:acylindrically} is proved by means of a general result on acylindrical actions on two-dimensional $\mathrm{CAT}(-1)$ spaces. Recall that an action of a group $G$ on a metric space $X$ is \textit{weakly acylindrical} \cite{MartinAcylSquare} if there exist constants $L, N \geq 0$ such that  two points of $X$ at distance  $\geq L$ are fixed by at most $N$ elements of $G$. Weak acylindricity is a dynamical condition that is weaker and much easier to deal with than acylindricity, especially for actions on non-locally compact spaces. Weak acylindricity was already known to be equivalent to acylindricity for actions on trees, and more generally for actions on finite-dimensional $\mathrm{CAT}(0)$ cube complexes \cite{ConingOffGenevois}. The following theorem, which is the central result of this article, is thus a powerful tool to study the dynamics of actions on two-dimensional $\mathrm{CAT}(-1)$ spaces (the result extends in a straightforward manner to $\mathrm{CAT}(\kappa)$ spaces, $\kappa < 0$).

\begin{theoremalph}\label{thm:main_acyl}
Let $G$ be a group acting by simplicial isometries on a two-dimensional piecewise hyperbolic $\mathrm{CAT}(-1)$ simplicial complex $Y$ with finitely many isometry types of simplices. If the action of $G$ on $Y$ is weakly acylindrical, then it is acylindrical.
\end{theoremalph}

\paragraph{Organisation of the article.} In Section~\ref{sec:preliminaries}, we recall a few basic facts about dihedral Artin groups, whose properties are used to understand the links of vertices of the modified Deligne complex $\Phi$ in a two-dimensional Artin group. In Section~\ref{sec:complex}, we recall the definition of $\Phi$ and endow it with a particular $\mathrm{CAT}(-1)$ metric. In Section~\ref{sec:trees}, we introduce the standard trees, which are the main obstruction to the acylindricity of the action of $A_S$ on $\Phi$. We describe their geometry and use it to construct the coned off space $\Phi^*$, which we show to admit a $\mathrm{CAT}(-1)$ metric. Section~\ref{sec:dynamics} is devoted to the proof of acylindricity Theorems~\ref{thm:main_acyl} and~\ref{thm:acylindrically}, and relies on a fine control of geodesics in a two-dimensional $\mathrm{CAT}(-1)$ simplicial complex. With the acylindricity of the action of $A_S$ on $\Phi^*$, we are then able to prove Corollaries~\ref{thm:Tits Alternative}--\ref{thm:classification_prod}.

\paragraph{Acknowledgements.} We thank Florestan Brunck and the referees for helpful remarks.

\section{Preliminaries on dihedral Artin groups}
\label{sec:preliminaries}

Let $S=\{s,t\}$ with $m=m_{st}<\infty$, and let us consider the dihedral Artin group~$A_S$.  In this section, we recall a few facts about $A_S$. The following is well known, see for example \cite[Lem 4.3(1)]{HJP} for a proof.

\begin{lem}
\label{lem:vertexgroup}
$A_S$ has a finite index subgroup isomorphic to $\mathbb{Z}\times F$, where $F$ is a free group (non-abelian if $m \geq 3$).	
\end{lem}

We denote $\Delta_{st}=\underbrace{st\cdots}_{m}\in A_S$.

\begin{lem}[{\cite[Thm~4.21]{Del}}]
\label{lem:centre}
Let $m\geq 3$. The centre of $A_S$ is generated by $\Delta_{st}$ for $m$~even and by $\Delta^2_{st}$ for $m$ odd.
\end{lem}

We therefore denote $z_{st}=\Delta_{st}$ for $m$ even (including the case $m=2$) and $z_{st}=\Delta^2_{st}$ for $m$ odd.

\begin{lem}[{\cite[Lem 7(ii)]{CrispAut}}]
\label{lem:centralizer}
Let $m\geq 3$. For any $k\neq 0$ the centralizer in $A_S$ of $s^k$ is the rank~$2$ abelian group generated by $s$ and $z_{st}$.
\end{lem}

\begin{lem}[{\cite[Lem~6]{AS}}]
\label{lem:AS}
Let $n<m$. A word with $2n$ syllables (i.e.\ of form $s^{i_1}t^{j_1}\cdots s^{i_n}t^{j_n}$ with all $i_k,j_k\in \Z-\{0\}$) is non-trivial in $A_S$.
\end{lem}

\section{Modified Deligne complex and its geometry}
\label{sec:complex}
Let $A_S$ be a two-dimensional Artin group. For $s,t\in S$ satisfying $m_{st}<\infty$, let $A_{st}$ be the dihedral Artin group with $m=m_{st}$. For $s\in S$, let $A_s=\Z$.

Let $K$ be the following simplicial complex. The vertices of $K$ correspond to subsets $T\subseteq S$ satisfying $|T|\leq 2$ and, in the case where $|T|=2$ with $T=\{s,t\}$, satisfying $m_{st}<\infty$. We call $T$ the \emph{type} of its corresponding vertex. Vertices of types $T,T'$ are connected by an edge of $K$, if we have $T\subsetneq T'$ or vice versa. Similarly, three vertices span a triangle of $K$, if they have types $\emptyset, \{s\}, \{s,t\}$ for some $s,t\in S$.

We give $K$ the following structure of a simple complex of groups $\mathcal K$ (see \cite[\S II.12]{BH} for background). The vertex groups are trivial, $A_s$, or $A_{st}$, when the vertex is of type $\emptyset,\{s\},\{s,t\}$, respectively. For an edge joining a vertex of type $\{s\}$ to a vertex of type $\{s,t\}$, its edge group is $A_s$; all other edge groups and all triangle groups are trivial. All inclusion maps are the obvious ones. It follows directly from the definitions that $A_S$ is the fundamental group of $\mathcal K$.

\subsection{Modified Deligne complex}
\label{subs:Deligne}
Assume now that $A_S$ is of hyperbolic type. We will equip $K$ with a $\mathrm{CAT}(-1)$ metric, which is inspired by the $\mathrm{CAT}(-1)$ construction of Moussong for Coxeter groups \cite[\S13]{Mou}. However, our construction is new. In particular, we make  some of the angles larger than Moussong, in order to prepare the construction of the coned off space in Definition~\ref{def:cone_off}.

Let $\eps>0$ be small enough so that for all $m\geq 2$ we have
\begin{equation*}
\Big(\frac{\pi}{2}+\eps\Big)+\Big(\frac{\pi}{2m}+\eps\Big)\leq\pi-\eps.
\tag{$\spadesuit$}
\end{equation*}
For any $m\geq 2$, let $\theta(m,\eps)$ be the third angle of any Euclidean triangle with angles $\frac{\pi}{2}+\eps,\frac{\pi}{2m}+\eps$. In other words, $\theta(m,\eps)=\frac{(m-1)\pi}{2m}-2\eps\geq \eps$.

\begin{lem}
\label{lem:eps}
There exists $\eps$ satisfying $(\spadesuit)$ and such that for any cycle $(s_i)$ in the defining graph of any Artin group of hyperbolic type, for $m_i=m_{s_is_{i+1}}$, we have
$\sum_i(\theta(m_i,\eps)-\eps)\geq \pi+\eps$.
\end{lem}

\begin{proof}
\textbf{Step 1.} There is $\eps$ such that for each hyperbolic triangle group with exponents $(m,m',m'')$ we have
$\theta(m,\eps)+\theta(m',\eps)+\theta(m'',\eps)\geq \pi+4\eps$.

In order to prove this, let $\eps$ be such that $\frac{1}{m}+\frac{1}{m'}+\frac{1}{m''}\leq1-24\frac{\eps}{\pi}$ for $(2,3,7),(2,4,5)$ and $(3,3,4)$ triangle groups. Since the exponents $(m,m',m'')$ of any hyperbolic triangle group dominate the exponents of one of these three, the same inequality holds for all hyperbolic triangle groups. Consequently,
\begin{align*}
\theta(m,\eps)+\theta(m',\eps)+\theta(m'',\eps)=&\frac{(m-1)\pi}{2m}+\frac{(m'-1)\pi}{2m'}+\frac{(m''-1)\pi}{2m''}-6\eps=\\
&\frac{3\pi}{2}-\Big(\frac{\pi}{2m}+\frac{\pi}{2m'}+\frac{\pi}{2m''}\Big)-6\eps\geq \\ &\frac{3\pi}{2}-\Big(\frac{\pi}{2}-12\eps\Big)-6\eps\geq \pi+6\eps.\\
\end{align*}

\noindent \textbf{Step 2.} There is $\eps$ such that for any $4$-cycle we have $\sum_i(\theta(m_i,\eps)-\eps)\geq \pi+ \eps$.

Indeed, since $A_S$ is of hyperbolic type, at least one $m_i$ is $\geq 3$, and consequently $\sum_i\theta(m_i,\eps)\geq 3\frac{\pi}{4}+\frac{\pi}{3}-8\eps$. Hence it suffices to take $(8+4+1)\eps\leq \frac{\pi}{12}$.

\smallskip

\noindent \textbf{Step 3.} There is $\eps$ such that for any $k$-cycle with $k\geq 5$ we have $\sum_i(\theta(m_i,\eps)-\eps)\geq\eps$.

Indeed, we have $\sum_i\theta(m_i,\eps)\geq 5(\frac{\pi}{4}-2\eps)$. Hence it suffices to take $(10+5+1)\eps\leq\frac{\pi}{4}$.
\end{proof}

From now on, we fix any $\eps$ satisfying Lemma~\ref{lem:eps}.

\begin{lem}
\label{lem:r}
Given a finite set $M\subset \{2,3,\ldots\}$, for sufficiently small $l>0$ we have that for any $m\in M$ there exists a hyperbolic triangle $vv'v''$ with angles $\an v''=\frac{\pi}{2m}+\eps, \an v'=\frac{\pi}{2}+\eps, \an v \geq\theta(m,\eps)-\eps,$ and $|vv'|=l$ (see Figure~\ref{fig:triangle}).
\end{lem}

\begin{proof}
We claim that for any $l>0$ and any $m\in \{2,3,\ldots\}$ there is a hyperbolic triangle $vv'v''$ with
$\an v''=\frac{\pi}{2m}+\eps, \an v'=\frac{\pi}{2}+\eps$, and $|vv'|=l$. Indeed, fix $v,v'$ at distance $l$ and a half-line at angle $\frac{\pi}{2}+\eps$ to $vv'$ at $v'$. Varying $v''$ along that half-line we can achieve any angle at $v''$ between $0$ and $\frac{\pi}{2}-\eps$. By $(\spadesuit)$, we have
$\frac{\pi}{2m}+\eps< \frac{\pi}{2}-\eps$, justifying the claim.

Denote $d(l,m)=|v'v''|$. It is easy to see that for fixed $m\in M$ and $l\to 0$, we have $d(l,m)\to 0$. Thus for sufficiently small $l$, the area of the triangle $vv'v''$ is arbitrarily small, hence by the Gauss--Bonnet Theorem so is its defect and thus $\an v \geq \theta(m,\eps)-\eps$. Since $M$ is finite, for sufficiently small $l$ this holds for all $m\in M$ simultaneously.
\end{proof}

To choose $l$ appropriately we need the following, the role of which will become clear in Section~\ref{sec:trees}.

\begin{rem}
\label{rem:cone}
In a right-angled hyperbolic triangle with legs of lengths $1$ and $d$, for $d$ sufficiently small, the other angle at the length $d$ leg is $\geq \frac{\pi}{2}-\eps$.
\end{rem}

We fix arbitrary $l>0$ satisfying Lemma~\ref{lem:r} and with $d=d(l,m)$ satisfying Remark~\ref{rem:cone},
for all exponents $m=m_{st}$ of $A_S$. We now equip $K$ with a piecewise hyperbolic metric. Let $\tau$ be a triangle of $K$ with vertices $v,v',v''$ of types $\emptyset,\{s\},\{s,t\}$, respectively. We equip $\tau$ with the metric of the unique hyperbolic triangle from Lemma~\ref{lem:r}, with $m=m_{st}$.

\begin{figure}[H]
	\begin{center}
		\scalebox{0.75}{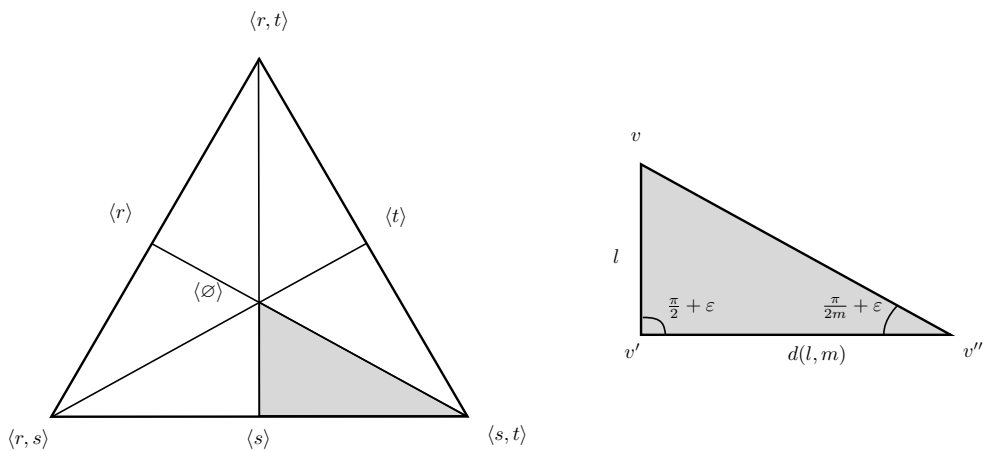}
		\caption{On the left-hand side, a picture of the fundamental domain $K$ for an Artin group on three generators $s, t, r$, with all $m_{st},m_{sr},m_{tr}<\infty$. The triangle spanned on the vertices of type $\emptyset, \{s\}$, and $\{s, t\}$ is shaded. On the right-hand side, the corresponding choice of lengths and angles for that triangle.}
		\label{fig:triangle}
	\end{center}
\end{figure}

Note that this choice is consistent on the edges, since two triangles sharing an edge either share a vertex of type $\{s,t\}$ and are thus congruent, or they share an edge with vertices of types $\emptyset,\{s\}$, which is of common length $l$.

For each vertex $v$ of $K$, let $\mathrm{St}(\tilde{v})$ be the local development at $v$ of~$\mathcal{K}$. The vertex of $\mathrm{St}(\tilde{v})$ corresponding to $v$ is labelled $\tilde{v}$. See \cite[\S~II.12.24]{BH}.

\begin{lem}
\label{lem:nonpositive_curvature}
For each vertex $v$ of $K$, the link of $\tilde v$ in $\mathrm{St}(\tilde{v})$ has girth $\geq 2\pi+\eps$.
\end{lem}

\begin{proof}
As in \cite{CD}, the idea is to appeal to Lemma~\ref{lem:AS}.

At $v$ of type $\emptyset$, $\mathrm{St}(\tilde{v})$ coincides with $K$. The link of $v$ in $K$ coincides with the barycentric subdivision of the defining graph with the length of the edge $(s,t)$ being $\geq 2(\theta(m,\eps)-\eps)$ (Lemma~\ref{lem:r}). Hence the lemma follows from Lemma~\ref{lem:eps}.

Suppose now that $v$ has type $\{s\}$. Then the link of $\tilde v$ in $\mathrm{St}(\tilde{v})$ is a bipartite graph of edge length $\frac{\pi}{2}+\eps$ and the lemma follows as well.

Finally, suppose that $v$ has type $\{s,t\}$ with $m=m_{st}$. Then the link of $\tilde v$ in $\mathrm{St}(\tilde{v})$
is the barycentric subdivision of the following graph $D$. Namely, consider the edge of groups $\mathcal I$ with vertex groups $A_s$ and $A_t$, trivial edge group, and the length of the underlying edge $2(\frac{\pi}{2m}+\eps)$. Consider the obvious morphism of $\mathcal I$ into $A_{st}$. Then $D$ is the development of $\mathcal I$ associated to that morphism. Thus it suffices to show that $D$ has girth $2m$. This is exactly Lemma~\ref{lem:AS}.
\end{proof}

By~\cite[Thm~II.12.28]{BH} we obtain the following (which is also a consequence of \cite[Thm 4.13]{L}).

\begin{cor}
$\mathcal{K}$ is strictly developable.
\end{cor}

\begin{defin}
\label{def:Deligne}
The development of $\mathcal{K}$ is called the \emph{modified Deligne complex} \cite{CD} and is denoted $\Phi$.
\end{defin}

Note that $\Phi$ is a triangle complex with a cocompact action of $A_S$. Its vertex stabilisers are trivial or conjugates of $A_s$ and~$A_{st}$, depending on their type, and its edge stabilisers are trivial or conjugates of~$A_s$.
In particular, all $A_s$ and $A_{st}$ with $m_{st}<\infty$ map injectively into $A_S$. Furthermore, $\Phi$ has finitely many isometry types of simplices and thus it is complete by \cite[Thm~I.7.19]{BH}. By \cite[Thm~II.12.28]{BH}, $\Phi$
is $\mathrm{CAT}(-1)$. Vertices of $\Phi$ inherit types from the types of the vertices of $K$.

\subsection{Non-hyperbolic case}
\label{sec:nonhyp}
Here we drop the hypothesis that $A_S$ is of hyperbolic type. Let $\mathcal K$ be the same complex of groups as before with the metric on each triangle being Euclidean with angles $\frac{\pi}{2}, \frac{\pi}{2m_{st}}, \frac{(m_{st}-1)\pi}{2m_{st}}=\theta(m_{st},0)$. Setting $\eps=0$, the same arguments as before give that the local developments of $\mathcal K$ are $\mathrm{CAT}(0)$ and hence $\mathcal K$ is strictly developable and its development $\Phi$ exists and is $\mathrm{CAT}(0)$. See \cite{CD} for detailed proof and the description of this piecewise Euclidean \emph{Moussong metric} in general.

\section{Standard trees and the coned off space $\Phi^*$}
\label{sec:trees}

Let $A_S$ be a two-dimensional Artin group, possibly not of a hyperbolic type. Let $\Phi_\mathcal{T}\subset \Phi$ be the subcomplex that is the union of all the edges of $\Phi$ joining vertices of type $\{s,t\}$ and $\{s\}$ for all $s,t\in S$.
Let $r\in S$ and let $T$ be the fixed-point set in $\Phi$ of $r$. Note that since $A_S$ acts on $\Phi$ without inversions, $T$ is a subcomplex of~$\Phi$. Since the stabilisers of the simplices of $\Phi$ outside $\Phi_\mathcal{T}$ are trivial, we have that $T\subset \Phi_\mathcal{T}$. In particular $T$ is a graph. Since $\Phi$ with the Moussong metric is $\mathrm{CAT}(0)$, $T$ is convex and thus it is a tree.

\begin{defin}\label{def:standard_tree}
A \textit{standard tree} is the fixed-point set in $\Phi$ of a conjugate of a generator $r\in S$ of $A_S$.
\end{defin}

The first goal of this section is to describe standard trees and their stabilisers, in the spirit of Example~\ref{exa:tree} ahead. In particular, see Figure~\ref{fig:standard}.

Recall from Section~\ref{sec:nonhyp} that, in the Moussong metric, the angles of triangles at a vertex of type $\{s,t\}$ are $\frac{\pi}{2m_{st}}$. From the convexity of $T$ we thus have immediately:

\begin{cor}
\label{cor:tree_local}
Let $T$ be a standard tree, and let $v$ be a vertex of $T$ incident to edges $e,e'$ of $T$. Then the combinatorial distance between their corresponding vertices in the link of $v$ is at least $2m_{st}$ for $v$ of type $\{s,t\}$, or exactly $2$ for $v$ of type~$\{s\}$. Consequently, in the case where $A_S$ is of hyperbolic type, their distance in the angular metric induced from the piecewise hyperbolic metric is at least $\pi+2m_{st}\eps$ for $v$ of type $\{s,t\}$ or exactly $\pi+2\eps$ for $v$ of type $\{s\}$.
\end{cor}

The following lemma will allow us to describe the structure of the stabiliser of a standard tree. For a vertex $v$ of type $\{s,t\}$, with $v=gv_K$ for $g\in A_S$ and $v_K$ the unique vertex of type $\{s,t\}$ in $K$, we define $z_v= gz_{st}g^{-1}$. Note that $z_v$ does not depend on $g$, since if $g'v_K=gv_K$, then $g^{-1}g'\in A_{st}$ and hence $g^{-1}g'$ commutes with $z_{st}$ implying $gz_{st}g^{-1}=g(g^{-1}g')z_{st}(g^{-1}g')^{-1}g^{-1}=g'z_{st}(g')^{-1}$.

\begin{lem}
\label{lem:tree_local}
Let $e,e'$ be edges in $\Phi$ with a common vertex $v$ of type $\{s,t\}$. Then either $\St(e)\cap\St(e')=\{\Id\}$ or $\St(e)=\St(e')$. Moreover in the latter case,
\begin{itemize}
\item
if $e,e'$ are of the same type, then there is $g\in \langle z_v\rangle$ with $e'=ge$, and
\item
if $e$ corresponds to the coset $A_s$ and $e'$ to the coset $h'A_t$, then $m_{st}$ is odd and there is $h\in \Delta_{st}\langle z_{st}\rangle$ satisfying $hA_t=h'A_t$.
\end{itemize}
\end{lem}

\begin{proof}
Assume without loss of generality that $e$ corresponds to the trivial coset~$A_s$. Then $\St(v)=A_{st}$. Assume first that $e'$ corresponds to a coset $h'A_s$.
Note that whenever we will establish $e'=ge$ for some $g\in \langle z_{st}\rangle$, we will have $\St(e)=\St(e')$. If $m=2$, then $\langle z_{st}\rangle A_s=A_{st}$. Thus there is $g\in \langle z_{st}\rangle$ with $gA_s=h'A_s$, and so $e'=ge$, as desired. Suppose now $m\geq 3$. If $\St(e)\cap\St(e')\neq\{\Id\}$, then we have $h's^k(h')^{-1}e=e$ for some $k>0$. This means that $h's^k(h')^{-1}\in A_s$, and using the homomorphism $A_{st}\to \Z$ mapping both generators to $1$ we obtain $h's^k(h')^{-1}=s^k$. By Lemma~\ref{lem:centralizer}, there is $g\in \langle z_{st}\rangle$ with $h'A_s=gA_s$, as desired.

If $e'$ corresponds to $h'A_t$ and $\St(e)\cap\St(e')\neq\{\Id\}$, then we have $h't^k(h')^{-1}e=e$ for some $k>0$. This means that $h't^k(h')^{-1}\in A_s$, and using the same homomorphism $A_{st}\to \Z$ we obtain $h't^k(h')^{-1}=s^k$. If $m_{st}$ is even, then $s$ and $t$ are not conjugate (use a homomorphism to $\Z$ killing $s$ but not $t$), contradiction. When $m_{st}$ is odd, we have $s\Delta_{st}=\Delta_{st}t$ and consequently $s^k=\Delta_{st}t^k\Delta_{st}^{-1}$, so that $(h')^{-1}\Delta_{st}$ commutes with $t^k$. By Lemma~\ref{lem:centralizer} there is $h$ with $hA_t=h'A_t$ and $h^{-1} \Delta_{st}\in\langle z_{st}\rangle$. Thus $\St(e')=hA_th^{-1}=A_s=\St(e)$.
\end{proof}

\begin{rem}
\label{rem:trees unique}
By Lemma~\ref{lem:tree_local}, the stabilisers of all edges in a standard tree coincide. Consequently each edge of $\Phi_\mathcal{T}$ belongs to exactly one standard tree and each vertex of type $\{s\}$ belongs to exactly one standard tree.
(In contrast, each vertex of type $\{s,t\}$ belongs to infinitely many standard trees, for $m_{st}>2$, or to two standard trees, for $m_{st}=2$.)
\end{rem}

\begin{lem}
\label{lem:stab_tree}
The stabiliser of the standard tree that is the fixed-point set of $r\in S$ is of the form $A_r\times F$ for some free group $F$.
\end{lem}
\begin{proof}
Let $T$ be the standard tree that is the fixed-point set of $r\in S$.
Since $r$ fixes an edge of $T$, any element $g\in \St(T)$ conjugates $r$ to an element also fixing an edge of $T$, which must be
some $r^k$ by Remark~\ref{rem:trees unique}. In fact, we obtain $k=1$ using the homomorphism $A_S\to \Z$ mapping all the generators to $1$. Hence $A_r=\Z$ is in the centre of $\St(T)$. The quotient $F=\St(T)/A_r$ acts on $T$ with trivial edge stabilisers. Thus $F$ is the fundamental group of the quotient graph of groups $T/F$, whose edge groups are trivial and whose vertex groups are $\Z$ by Lemma~\ref{lem:tree_local}. Consequently $F$ is free. Taking any splitting $F\to \St(T)$ gives us $\St(T)=A_r\times F$.
\end{proof}

\begin{rem}
\label{rem:description}
We have the following explicit description of the stabiliser of the standard tree $T$ in Lemma~\ref{lem:stab_tree}. Let $\overline{T}$ be the graph obtained from $K_\mathcal{T}=K\cap \Phi_\mathcal{T}$ by cutting it along all the vertices of type $\{s,t\}$ with $m_{st}$ even. Vertices of $\overline{T}$ inherit types from the vertices of $K_\mathcal{T}$.
Let $\overline{T}_r$ be the component of $\overline{T}$ containing the unique vertex of type $\{r\}$.
By Lemma~\ref{lem:tree_local}, we have $T/F=\overline {T}_r$, with vertex groups~$\Z$ at all the vertices of type $\{s,t\}$.
Introduce an order $s_1,s_2,\ldots$ on the elements of~$S$. Label each directed edge $e$ of $\overline T_r$ connecting the vertex of type $\{s_i\}$ to the vertex of type $\{s_i,s_j\}$ with $i<j$ and $m_{s_is_j}$ odd with the element $\phi(e)=\Delta_{s_is_j}\in A_S$. Label all other edges of $\overline T_r$ by the trivial element. 
From Lemma~\ref{lem:tree_local} one can deduce that we can take $F$ freely generated by
\begin{itemize}
\item
the words labelling a set of closed paths in $\overline T_r$ based at the vertex of type~$\{r\}$ forming a free basis of $\pi_1\overline{T}_r$, and
\item
the conjugates of $z_{st}$ by the words labelling some paths in $\overline T_r$ joining the vertex of type $\{r\}$ with each of the vertices of type $\{s,t\}$ in $\overline{T}_r$ .
\end{itemize}
\end{rem}

\begin{ex}
\label{exa:tree}
We illustrate Remark~\ref{rem:description} on an example, see Figure~\ref{fig:standard} below. Let $S=\{s,t,r\}$ with $m_{st}=4, m_{tr}=2, m_{sr}=\infty$. Let $e$ be the unique edge of $K$ joining the vertices of types $\{s\}$ and $\{s,t\}$. Then the standard tree $T_s$ that is the fixed-point set of $s$ is the union of the translates of $e$ under $\langle z_{st} \rangle$, which is bounded. The stabiliser of $T_s$ is $\langle s, z_{st} \rangle \cong \Z^2$.

Let $f$ (respectively, $f'$) be the unique edge of $K$ joining the vertices of types $\{t\}$ and $\{s,t\}$ (respectively, $\{t\}$ and $\{t,r\}$). Then the standard tree $T_t$ that is the fixed-point set of $t$ is not bounded. Each of the vertices of $T_t$ of type $\{t\}$ has degree~$2$, and each of the vertices of $T_t$ of type $\{s,t\}$ or $\{t,r\}$ has infinite degree. The stabiliser of $T_t$ is the direct product of $A_t=\Z$ and the free group generated by $z_{st}$ and $z_{tr}$.

\begin{figure}[H]
	\begin{center}
		\scalebox{1}{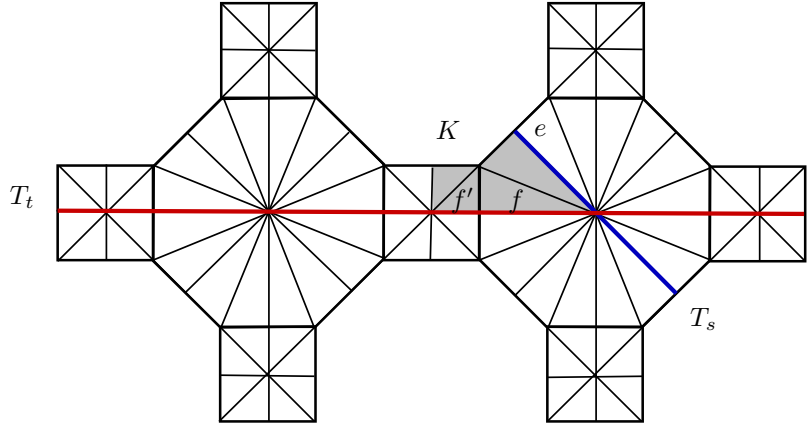}
		\caption{A (very small) portion of the Deligne complex for $A_S = \langle s, t, r ~|~ stst=tsts, rt=tr \rangle$, with the fundamental domain $K$ shaded. The bounded standard tree $T_s$ is represented in blue, while the unbounded standard tree $T_t$ is represented in red.}
		\label{fig:standard}
	\end{center}
\end{figure}

\end{ex}

\begin{defin}\label{def:cone_off}
Suppose now that $A_S$ is of hyperbolic type, and equip $\Phi$ with the $\mathrm{CAT}(-1)$ metric of Section~\ref{subs:Deligne}.
Let $\Phi^*$ be the $2$-complex obtained by coning off simplicially each of the standard trees. In $\Phi$ consider an edge of a standard tree with vertices $v',v''$ of type $\{s\},\{s,t\}$ respectively, and let $c$ be its cone vertex. We put on the triangle $cv'v''$ the metric of a right-angled hyperbolic triangle with the right angle at $v'$, $|cv'|=1$ and the length $d=|v'v''|$ (depending on $m_{st}$) as in $\Phi$ so that by Remark~\ref{rem:cone} the angle at $v''$ is $\geq \frac{\pi}{2}-\eps$. 

Since $\Phi$ has finitely many isometry types of simplices, $\Phi^*$ also has finitely many isometry types of simplices. In particular $\Phi^*$ is complete by \cite[Thm~I.7.19]{BH}.

 The action of $A_S$ on $\Phi$ extends to an action of $A_S$  on $\Phi^*$, where each $g\in A_S$ maps the cone vertex of a standard tree $T$ to the cone vertex of the standard tree $gT$. Since the metric on each triangle $cv'v''$ above depends only on the length of the edge $v'v''$ in $\Phi$, we have that $A_S$ acts on $\Phi^*$ by isometries.
\end{defin}

\begin{prop}
\label{prop:CAT(-1)}
For $\eps$ sufficiently small, $\Phi^*$ is $\mathrm{CAT}(-1)$.
\end{prop}

\begin{proof} Since $\Phi$ is simply connected and standard trees are simply connected, $\Phi^*$ is simply connected. $\Phi^*$ is piecewise hyperbolic, so by \cite[Thms~II.4.1(2) and II.5.24]{BH} it suffices to show that the link of each vertex $v$ is of girth $\geq 2\pi+\eps$. If $v$ is a cone point, its link is a tree and there is nothing to prove. The vertex links $L(v)$ in $\Phi$ are of girth $\geq 2\pi+\eps$ by Lemma~\ref{lem:nonpositive_curvature}. Hence if $v$ is of type $\emptyset$, we are done as well.

If $v$ if of type $\{s\}$, and $\alpha$ is a cycle in its link not contained in $L(v)$, then $\alpha$ passes through a vertex corresponding to an edge joining $v$ to a cone point, which is unique by Remark~\ref{rem:trees unique}. Hence $\alpha$ travels through two adjacent edges of length $\frac{\pi}{2}$, and through at least two edges in $L(v)$, which by Corollary~\ref{cor:tree_local} have length $\geq \frac{\pi}{2}+\eps$, as desired.

If $v$ if of type $\{s,t\}$, and $\alpha$ is a cycle in its link not contained in $L(v)$, then analogously $\alpha$ passes through a vertex corresponding to an edge joining $v$ to a cone point. Hence $\alpha$ travels through two adjacent edges of length $\geq \frac{\pi}{2}-\eps$. If the remaining part of $\alpha$ is contained in $L(v)$, then it suffices to use Corollary~\ref{cor:tree_local}. Finally, by Remark~\ref{rem:trees unique}, if $\alpha$ passes through exactly one other vertex (respectively, at least two other vertices) corresponding to a cone point, it is of length $\geq (2\pi-4\eps)+2(\frac{\pi}{m_{st}}+2\eps)$ (respectively, $\geq 3\pi-6\eps$), which is $\geq 2\pi+\eps$ for $\eps$ small enough with respect to $m_{st}$.
\end{proof}

\begin{conv}
\label{con}
From now on, we assume that $\varepsilon$ in Section~\ref{sec:complex} was chosen small enough so that the coned off space $\Phi^*$ is $\mathrm{CAT}(-1)$.
\end{conv}

\section{Dynamics of the action}
\label{sec:dynamics}
The goal of this section is to prove the acylindricity of the action of $G$ on $\Phi^*$ for two-dimensional Artin groups of hyperbolic type (Theorem~\ref{thm:acylindrically}) and its Corollaries~\ref{thm:Tits Alternative}--\ref{thm:classification_prod}.
In order to do this, we use the more general Theorem~\ref{thm:main_acyl} whose proof we postpone to the next subsection.

\subsection{Weak acylindricity of the action of $A_S$ on $\Phi^*$}
\label{sec:theorems}

\begin{lem}
	\label{lem:distancerealised}
	Let $Y$ be a connected piecewise hyperbolic simplicial complex with finitely many isometry types of simplices. Then for any non-empty subcomplexes $T,T'\subseteq Y$ there are points $x\in T,x'\in T'$ realising the distance between $T$ and $T'$.
	
	Moreover, if $Y$ is $\mathrm{CAT}(-1)$ and $T,T'$ are convex and disjoint, then such a pair of points $x,x'$ is unique.
\end{lem}

\begin{proof}  Let $d_0 = d(T,T')$. By \cite[Thm~I.7.28]{BH}, there is a constant $c$ such that each 
taut string in $Y$ \cite[Def~I.7.20]{BH} of length $\leq d_0+1$ has size $\leq c$. In particular, since each geodesic segment $\gamma$ in $Y$ of length  $\leq d_0+1$ determines a taut string of length  $\leq d_0+1$, we have that the minimal subcomplex $Gal(\gamma)$ of $Y$ containing $\gamma$ is the union of at most $c$ simplices. Since $Y$ has finitely many isometry types of simplices, there are only finitely many isometry types of such~$Gal(\gamma)$. Below, $d_{Gal(\gamma)}$ denotes the intrinsic distance function on each $Gal(\gamma)$.

To justify that $d(T,T')$ is realised, we will prove that $d(T,T')=\min d_{Gal(\gamma)}(\sigma,\sigma')$ over all simplices $\sigma\subseteq T, \sigma'\subseteq T'$ and geodesics $\gamma$ of length $\leq d_0+1$ between some points in $\sigma$ and $\sigma'$. Indeed, we have $d(T,T')\leq d(\sigma,\sigma')\leq d_{Gal(\gamma)}(\sigma,\sigma')$. On the other hand, $d(T,T')$ is the infimum of the lengths of $\gamma$, which are $\geq d_{Gal(\gamma)}(\sigma,\sigma')$.
	
 The second assertion follows from the strict convexity of the distance function in $\mathrm{CAT}(-1)$ spaces (see for example \cite[Prop~II.2.2]{BH}, where for $X$ a $\mathrm{CAT}(-1)$ space and $c(0)\neq c'(0)$ one obtains the strict inequality). \end{proof}

\begin{prop}\label{cor:w_acyl}
	The action of $A_S$ on $\Phi^*$ is weakly acylindrical.
\end{prop}

\begin{proof}
	Suppose $g\in A_S$ is a non-trivial element fixing points $y,y'\in \Phi^*$. To prove weak acylindricity it suffices to bound the distance $d(y,y')$ from above. Let $v$ (resp.~$v'$) be a vertex of the simplex containing $y$ (resp.\ $y'$) in its interior. We will bound from above the combinatorial distance between $v$ and $v'$.
	
	Since $A_S$ acts without inversions, $g$ fixes $v,v'$. We now define the following subcomplexes $T,T'$ of~$\Phi$. If $v$ is a vertex of $\Phi$, set $T=v$. If $v$ is a cone vertex, set $T$ to be the standard tree corresponding to $v$. Define $T'$ analogously.
We can assume that $T,T'$ are disjoint, since otherwise $v,v'$ are at combinatorial distance $\leq 2$ in $\Phi^*$, as desired.
 By Lemma~\ref{lem:distancerealised} applied to $Y=\Phi$, there are unique points $x\in T$ and $x'\in T'$ realising the distance between $T$ and $T'$. Let $\gamma$ be the geodesic of $\Phi$ between $x$ and~$x'$. Since $g$ stabilises $T$ and $T'$, we have that $g$ fixes~$\gamma$.
	
	If $\gamma$ is not contained in $\Phi_{\mathcal T}$, then it has a point with trivial stabiliser, contradicting the assumption that $g$ is non-trivial. Assume now that $\gamma$ is contained in~$\Phi_{\mathcal T}$. Suppose that $\gamma$ has two consecutive edges $e, e'$ that belong to distinct standard trees. By Lemma \ref{lem:tree_local}, we have that $\St(e)\cap\St(e')$ is trivial, contradicting again the assumption that $g$ is non-trivial. In the remaining case, the entire path $\gamma$ is contained in one standard tree. Consequently $v,v'$ are at combinatorial distance $\leq 4$ in $\Phi^*$, as desired.
\end{proof}

\begin{rem}
\label{rem:elliptic}
If $H\subset A_S$ acts elliptically on $\Phi^*$, then since $\Phi^*$ is $\mathrm{CAT}(-1)$ and complete, by \cite[Thm~II.2.8(1)]{BH} $H$ fixes a point $v$ of~$\Phi^*$. Since $A_S$, and hence~$H$, acts without inversions, we can take $v$ a vertex, which is a vertex of $\Phi$ or a cone vertex. Thus by Lemmas~\ref{lem:vertexgroup} and~\ref{lem:stab_tree} $H$
has a finite index subgroup contained in $\Z\times F$ for a free group $F$.
\end{rem}

\begin{proof}[Proof of Theorem~\ref{thm:acylindrically}] Since the action of $A_S$ on $\Phi^*$ is weakly acylindrical by Proposition~\ref{cor:w_acyl}, it follows from Theorem~\ref{thm:main_acyl} that the action is acylindrical. Assume now that for each $s\in S$ there is $t\in S$ with $m_{st}< \infty$ and let us show the universality of this action. By a theorem of Bridson \cite[Thm A]{B}, a simplicial isometry $g$ of a piecewise hyperbolic complex with finitely many isometry types of simplices is either loxodromic or elliptic. In particular, this applies to the action of each $g\in A_S$ on $\Phi^*$. Thus if $g$ is not loxodromic, by Remark~\ref{rem:elliptic} applied to $H=\langle g \rangle$, there is $k>0$ with $g^k\in \Z\times F\subset A_S$ for a free group $F$. The group $F$ has rank $\geq 1$ since in the case where $g$ stabilises a standard tree that is a translate of the fixed-point set of $s\in S$, we assumed that there is $t\in S$ with $m_{st}<\infty$. Thus $g^k$ generates an infinite cyclic subgroup with infinite index in its centraliser. It follows from \cite[Cor 6.9]{OsinAcyl} that $g$ cannot be generalised loxodromic.

Finally, assume that $|S|\geq 3$ and $A_S$ is irreducible. Since $|S|\neq 1$, we have that $A_S$ is not virtually cyclic. Since the action of $A_S$ on the hyperbolic space $\Phi^*$ is acylindrical, to show that $A_S$ is acylindrically hyperbolic it suffices to prove that this action is not elliptic. Indeed, otherwise by Remark~\ref{rem:elliptic} applied with $H=A_S$, we have that $A_S$ fixes a vertex of $\Phi$ or a cone vertex. If $A_S$ fixes a vertex of type $\{s\}$ or $\{s,t\}$, then we have $A_S=A_s$ or $A_S=A_{st}$, which contradicts $|S|\geq 3$. If $A_S$ fixes a cone vertex corresponding to a standard tree that is, say, the fixed point set of $s\in S$, then by Lemma~\ref{lem:stab_tree} the group $A_S$ centralises $s$. Thus for all $t\in S-\{s\}$ we have $m_{st}=2$, which contradicts the irreducibility of $A_S$.
\end{proof}

Note that the condition that for each $s\in S$ there is $t\in S$ with $m_{st}<\infty$ is necessary for the action to be universal. Indeed, otherwise $A_S=A_s*A_{S-\{s\}}$ and $\Phi^*$ is a tree of spaces with vertex spaces of two types. The first type are the coned off modified Deligne complexes for $A_{S-\{s\}}$. The second type are edges joining the fixed points of the conjugates of $s$ with their cone points. These vertex spaces are joined by edges with vertices of type $\{s\}$ and $\emptyset$. If we replace the vertex spaces of the second type with real lines containing a $\Z$'s worth of vertices of type $\{s\}$, the action stays acylindrical but $s$ becomes loxodromic. Note also that after performing these replacements for all free factors $A_s$ of $A_S$ we obtain a universal acylindrical action.

Note also that if $A_S$ is not irreducible, then it is not acylidrically hyperbolic \cite[Cor~7.3(b)]{OsinAcyl}. 
If $S=\{s,t\}$ and $m_{st}=\infty$, then $A_S$ is the free group on $s$ and~$t$, so it is acylindrically hyperbolic. Finally, if  
$S=\{s,t\}$ and $m_{st}<\infty$, then the group $A_S$ is virtually $\Z\times F$ for some free group $F$ (Lemma~\ref{lem:vertexgroup}). Thus $A_S$ is not acylidrically hyperbolic \cite[Cor~7.3(b)]{OsinAcyl}.

\begin{proof}[Proof of Corollary \ref{thm:Tits Alternative}]
Let $H$ be a subgroup of $A_S$ that is not virtually cyclic. Since the action of $A_S$ on the hyperbolic space $\Phi^*$ is acylindrical by Theorem~\ref{thm:acylindrically}, $H$ is elliptic or acylindrically hyperbolic. If $H$ is elliptic, then by Remark~\ref{rem:elliptic} $H$ is virtually contained in $\Z\times F$ for a free group $F$, and thus it is virtually $\Z^2$ or contains a non-abelian free group. If $H$ is acylindrically hyperbolic, then it contains a non-abelian free group by \cite[Thm~6.14]{DGO}.
\end{proof}

\begin{proof}[Proof of Corollary \ref{thm:classification_Z2}]
Let $H$ be a subgroup of $A_S$ that is virtually $\mathbb{Z}^2$. Since $H$ is not acylindrically hyperbolic by \cite[Cor~7.3(b)]{OsinAcyl}, it is elliptic by Theorem~\ref{thm:acylindrically}.
By Remark~\ref{rem:elliptic}, $H$ stabilises a vertex of $\Phi$ or a cone vertex of $\Phi^*$, and hence a standard tree of $\Phi$. In the latter case by Lemma~\ref{lem:stab_tree}, $H$ is contained in a $\Z\times F$ with $F$ a free group, and $\Z$ conjugate to some $A_s$, as desired.
\end{proof}

\begin{proof}[Proof of Corollary \ref{thm:classification_prod}]
Let $H$ be a subgroup of $A_S$ that is virtually a non-trivial direct product. Since $A_S$ is torsion-free \cite[Thm~B]{CD}, by \cite[Cor~7.3(b)]{OsinAcyl} $H$ is not acylindrically hyperbolic, and thus it is elliptic by Theorem~\ref{thm:acylindrically}. By Remark~\ref{rem:elliptic}, $H$ is virtually of form $\Z\times F$, as desired.\end{proof}

\subsection{The general acylindricity theorem}

We now turn to the proof of Theorem \ref{thm:main_acyl}. In this section, $Y$ denotes a two-dimensional piecewise hyperbolic simplicial complex, with finitely many isometry types of simplices, and that is $\mathrm{CAT}(-1)$.\\

\paragraph{Simplifications.} We first explain how to alter the metric $d$ on $Y$ so that girths of vertex links become uniformly greater than $2\pi$. Replace every hyperbolic triangle of~$Y$ with side lengths $a, b, c$ by a hyperbolic triangle of side lengths $\frac{a}{2}, \frac{b}{2}, \frac{c}{2}$ respectively, and call this new metric $d_2$. With respect to $d_2$ all the angles are strictly larger than with respect to $d$. Since $(Y,d_2)$ still has finitely many isometry types of simplices, the girths of vertex links are now uniformly greater than $2\pi$. Therefore, without loss of generality we assume from now on that the same was true with respect to $d$. Note that in the case where $Y$ is the coned off space $\Phi^*$, it follows from the proof of Proposition \ref{prop:CAT(-1)} that vertex links already satisfy this condition to start with.

Secondly, we explain how to subdivide the complex so that all the triangles become acute. Namely, any finite piecewise euclidean triangle complex admits an \emph{acute} triangulation, i.e.\ a triangulation all of whose triangles are acute \cite{BZ}. While the piecewise hyperbolic counterpart of that result does not seem to appear in the literature, we can make use of the piecewise euclidean statement in the following way.

\begin{lem}
\label{lem:acute}
Let $Y$ be a two-dimensional piecewise hyperbolic simplicial complex, with finitely many isometry types of simplices. For $n\geq 1$, let $d_n$ be the metric on~$Y$ obtained by replacing every hyperbolic triangle of~$Y$ with side lengths $a, b, c$ by a hyperbolic triangle of side lengths $\frac{a}{n}, \frac{b}{n}, \frac{c}{n}$. There exists $n$ such that $(Y, d_n)$ admits an acute triangulation.
\end{lem}
\begin{proof}
For $a\geq b \geq c>0$ satisfying $a<b+c$, let $T$ be the Euclidean triangle with side lengths $a,b,c$, and let $T_n$ be the hyperbolic triangle of side lengths $\frac{a}{n}, \frac{b}{n}, \frac{c}{n}$. Let $nT_n$ denote $T_n$ with metric rescaled by the factor $n$. In other words, $nT_n$ is isometric to a triangle of side lengths $a,b,c$ in the hyperbolic plane rescaled by the factor $n$, hence of curvature $-\frac{1}{n^2}$.

For each $n\geq 1$, let $\psi_n\colon T_n\to D$ be an isometric embedding of $T_n$ in the unit disc~$D$ of $\R^2$ equipped with the Riemannian metric of the Klein model of the hyperbolic plane. We require additionally that $\psi_n(T_n)$ contains the centre $(0,0)$ of $D$. Let $\varphi_n\colon nT_n\to T$ be the diffeomorphism that is the composition of the rescaling map $nT_n\to T_n$, the map $\psi_n$, and the affine map sending $\psi_n(T_n)\subset \R^2$ to~$T$ (respecting the sides). Note that the Riemannian metrics $\sigma_n$ on $T$ pushed forward from $nT_n$ via~$\varphi_n$ converge pointwise to the euclidean Riemannian metric of~$T$, which we call~$\sigma_\infty$.
 Furthermore, the geodesics for $\sigma_{\infty}$
 coincide (up to a reparametrisation) with the geodesics for each $\sigma_{n}$. In particular, the unit tangent vector at $x$ to the unique geodesic $xy$ in the metric $\sigma_n$ converges to the unit tangent vector at $x$ to the unique geodesic $xy$ in the metric~$\sigma_\infty$.

Let $(Y,d_\infty)$ be the piecewise euclidean simplicial complex obtained from~$Y$ by replacing each hyperbolic triangle by the euclidean triangle $T$ with the same side lengths. Since $(Y,d_\infty)$ has finitely many isomorphism types of simplices, it admits a quotient map to a finite piecewise euclidean triangle complex $Y'$ that is an isometry on each of the simplices.
By \cite{BZ}, there is an acute triangulation~$\Delta'$ of $Y'$, which we pull back to an acute triangulation~$\Delta$ of $(Y,d_\infty)$. Then for $T$ of a fixed isometry type, there are only finitely many possibilities for the restriction of $\Delta$ to $T$. By the previous paragraph, for $n$ sufficiently large, pulling back to $nT_n$ the vertices of $\Delta$ in $T$ via $\varphi_n$, and joining the same pairs as in $\Delta$ by geodesic segments in $nT_n$, gives an acute triangulation~$\Delta_n$ of $nT_n$. For $n$ sufficiently large we have that $\Delta_n$ is acute for all $T$ simultaneously.

To find an acute triangulation of some $(Y,d_n)$ it suffices to find an acute triangulation of $(Y,nd_n)$.
We regard $(Y,nd_n)$ as the union of $nT_n$, over $T$ in $(Y,d_\infty)$. We wish to piece together an acute triangulation of $(Y,nd_n)$ from~$\Delta_n$. Note, however, that for a vertex $v$ of $\Delta$ in the interior of an edge of two (or more) triangles $T,T'$ of $(Y,d_\infty)$, the preimages $w_n,w_n'$ of $v$ in the triangles $nT_n,nT'_n$ of $(Y,nd_n)$ under the maps $\varphi_n,\varphi'_n$ might not coincide. In other words, $\Delta_n$ and $\Delta_n'$ might not match on a common edge of $nT_n,nT'_n$. However, the distance between $\varphi_n'(w_n)$ and $\varphi_n'(w_n')=v$ converges to $0$ as $n$ converges to $\infty$. Hence replacing $w_n'$ by $w_n$ in $nT'_n$ (which we do simultaneously in all such configurations, of which there are finitely many up to an isometry), we still obtain an acute triangulation of $nT'_n$ for $n$ sufficiently large. These triangulations piece together to an acute triangulation of $(Y,nd_n)$.
\end{proof}

By Lemma~\ref{lem:acute}, without loss of generality we assume that all the triangles of $(Y,d)$ are acute. In particular, stars are convex, and the union of two triangles sharing an edge is convex.

\smallskip

\paragraph{Notation.}
Let $v$ be a vertex of $Y$. We denote by $\pi_v: Y - \{v\} \rightarrow lk(v)$ the map assigning to each $y\in Y-\{v\}$ the direction of the geodesic from $v$ to $y$ (which is well-defined since geodesics are unique).
The angular distance between two points $x, x' \in lk(v)$ will be denoted $\measuredangle_v(x, x')$. For $y,y'\in Y-\{v\}$ we extend this notation so that $\measuredangle_v(y, y')=\measuredangle_v(\pi_v(y), \pi_v(y')).$
For $k>0$, we define the metric $k$-neighbourhood of a subset $Y'$ of $Y$ as the set $$\mathcal{N}_k(Y') =\{y \in Y ~|~ d(y, Y') < k\}.$$
For a point $y\in Y$, by $\tau_y$ we denote the simplex of $Y$ containing $y$ in its interior.
For a simplex $\tau$ of $Y$, its open star $st(\tau)$ is the union of the interiors of the simplices containing $\tau$.

\begin{defin}
\label{def:alpha}
We fix a constant $\alpha>0$ such that:
	\begin{itemize}
		\item the link of every vertex of $Y$ has girth $\geq 2\pi + 4\alpha$,
		\item every triangle of $Y$ has all angles $\geq 12\alpha$.
	\end{itemize}
\end{defin}

\paragraph{Choosing an open cover.} A key tool in controlling geodesics of $Y$ will be to subdivide them in pieces that are easier to understand locally. This will be done by means of an appropriate cover of $Y$, which will take us some time to define. The cover will consist of an open set $U_\tau$ for each simplex $\tau$ of $Y$. We will also define a constant $\epsilon$ (distinct from $\eps$ in Convention~\ref{con}) that depends only on $Y$. Their main properties will be:

\begin{description}
\item[($*$)] Each $\mathcal {N}_\epsilon(U_\tau)$ is contained in $st(\tau)$.
\item[($\cap$)] If $U_\tau$ intersects $U_{\tau'}$, then $\tau$ contains $\tau'$ or vice versa.
\end{description}

To start with, since $Y$ has only finitely many isometry types of simplices, we can fix a constant $\epsilon_0 >0$ such that the balls
$$U_v=\mathcal {N}_{6\epsilon_0}(v)$$ around the vertices $v$ of~$Y$ satisfy $\mathcal {N}_{4\epsilon_0}(U_v)\subset st(v)$ and property ($\cap$) for $\tau,\tau'$ vertices. Before we proceed with the construction of the remaining elements of the cover, we need the following.

\begin{defin}\label{lem:continuity_proj}
We fix a constant $0<\epsilon\leq \epsilon_0$ such that for each vertex $v$ of~$Y$ and a pair of points $x, y \in Y - \mathcal N_{\epsilon_0}(v)$ with $d(x, y) \leq \epsilon$, we have $$ \measuredangle_{v}(x, y) \leq \alpha.$$
(Such a constant exists, since $\pi_v$ is Lipschitz on $Y - \mathcal N_{\epsilon_0}(v)$.)
\end{defin}

Now, for an edge $e=vw$ of $Y$, we define
$$U_e= \mathcal{N}_{4\epsilon}(e  -  U_v \cup U_w).$$
Note that for any point $x\in U_e$, we have $d(v,x)\geq d(v,Y-U_v)-4\epsilon\geq 6\epsilon_0-4\epsilon\geq 2\epsilon_0$.
Then by Definitions~\ref{def:alpha} and~\ref{lem:continuity_proj}, for any edge $f=vu$ with $u\neq w$, we have $\measuredangle_v(x,u)\geq 12\alpha-4\alpha>0$ (since for $x'\in e-U_v \cup U_w$ with $d(x,x')<4\eps$ we have $\measuredangle_v(x',u)\geq 12\alpha$ and $\measuredangle_v(x,x')\leq 4\alpha$). Consequently, $U_e$ is disjoint from $f$, and so we have property~($*$) for $\tau$ an edge. Similarly we have property~($\cap$) for $\tau,\tau'$ edges. Property~($\cap$) for $\tau$ an edge and $\tau'$ a vertex follows from $\mathcal {N}_{4\epsilon_0}(U_v)\subset st(v)$ for $v=\tau'$.

Note also that for points $v',w'\in e$ at distance $6\epsilon_0-2\epsilon$ from $v,w$, we have $\mathcal{N}_{2\epsilon}(vv')\subset U_v$ and $\mathcal{N}_{2\epsilon}(v'w')\subset U_e$, and consequently $\mathcal{N}_{2\epsilon}(e)\subset U_v\cup U_e\cup U_w$. This is why when for $\sigma$ a triangle of $Y$, we define $$U_\sigma=\mathcal{N}_{\epsilon} \big(\sigma-\bigcup_{\tau\subset \partial \sigma}U_\tau\big),$$ we have property ($*$) and consequently property ($\cap$) for $\tau$ a triangle and $\tau'$ arbitrary. Furthermore, again by Definitions~\ref{def:alpha} and~\ref{lem:continuity_proj} we have the following:

\begin{cor}\label{cor:bounded_below_angle} Let $v$ be a vertex of $Y$, and let $\tau, \tau'\neq v$ be simplices of  $Y$ containing~$v$ and such that neither of them is contained in the other. Then for every $x \in U_\tau$ and $y\in U_{\tau'}$, we have
	 $$\measuredangle_{v}(x, y) \geq 4\alpha.$$
\end{cor}

\paragraph{Galleries and extended galleries.}

\begin{defin}[Gallery]\label{def:gallery}
Let $\gamma$ be a geodesic segment in $Y$. We denote by~$Gal(\gamma)$ the minimal subcomplex of $Y$ that contains $\gamma$, and we call it the \textit{gallery} of~$\gamma$.  It is the union of all the simplices $\tau_y$ over $y\in \gamma$. \end{defin}

We now want to slightly enlarge $Gal(\gamma)$.

\begin{defin}[Extended gallery]\label{def:stable_gallery}
Let $\gamma$ be a geodesic segment in $Y$, let $Gal(\gamma)$ be its gallery, and let $V(\gamma)$ be the (possibly empty) set of vertices of $Y$ contained in~$\gamma$  and which are not an endpoint of $\gamma$. For each $v \in V(\gamma)$, we perform the following construction.

Since the geodesic $\gamma$ passes through $v$, $\gamma$ defines two points $x_1, x_2$ in the link $lk(v)$ at angular distance $\geq \pi$. Since vertex links have girth $\geq 2\pi + 4\alpha$, there exists at most one geodesic $\ell_v$ of $lk(v)$ (for the angular metric) of length $<\pi+2\alpha$ between $x_1$ and $x_2$. If no such geodesic $\ell_v$ exists, we set
$E_v = \emptyset.$ Otherwise, let $E_v$ be the set of edges of the minimal subgraph of $lk(v)$ containing $\ell_v$. Each edge $e$ of $E_v$ corresponds to a triangle $\sigma_e$ of $Y$ containing $v$.
We set $$Gal^{*}(\gamma) = Gal(\gamma) \cup  \bigcup_{v \in V(\gamma)} \bigcup_{e \in E_v} \sigma_e ,$$
which we call the \textit{extended gallery} of $\gamma.$
\end{defin}

\begin{rem}
\label{rem:extendedsize}
By \cite[Thm~I.7.28]{BH}, and since angles of triangles in $Y$ are bounded from below, there is a constant $C$ such that for each geodesic segment $\gamma$ of length $|\gamma|$, the extended gallery $Gal^{*}(\gamma)$ contains at most $C|\gamma|$ vertices.
\end{rem}

\paragraph{Appropriate subdivision of a geodesic.}

\begin{defin}[Decomposition]
Let $\gamma = xy$ be a geodesic in $Y$, oriented from $x$ to $y$. By property ($\cap$), we can choose a shortest sequence of simplices $(\Sigma_i)_{i=0}^{k+1}$ such that there are points $x_i\in U_{\Sigma_i}$ lying on $\gamma$ in that order, with $x_0=x,x_{k+1}=y$ and $\sigma_i=\Sigma_{i-1}\cap \Sigma_{i}\neq \emptyset$.
We call the data of all $\Sigma_i$ and $x_i$ (and $\sigma_i$ determined by them) a \textit{decomposition} of~$\gamma$.
\end{defin}

\begin{rem}
\label{rem:anchored} If in each pair $\Sigma_i, \Sigma_{i+1}$ none of the simplices is contained in the other, we say that the decomposition is \emph{anchored}. Note that by the minimality of~$k$, $\Sigma_i$ cannot be contained in $\Sigma_{i+1}$ unless $i=0$, and
$\Sigma_i$ cannot contain $\Sigma_{i+1}$ unless $i=k$. Thus the geodesic $x_1x_k$ has an anchored decomposition obtained by discarding $\Sigma_0, \Sigma_{k+1}$.\end{rem}

\paragraph{Technical lemma.} The following lemma is the key technical result allowing us to control the simplices met by a geodesic sufficiently close to another one. We advise the reader to skip its proof during a first reading.

\begin{lem}\label{lem:stable_global} Let $\gamma=xy,\gamma' = x'y'$ be two geodesics in $Y$ with $d(x,x'),d(y,y')<\epsilon$. Suppose that $\gamma$ has an anchored decomposition with $k\geq 0$. Then
	$$Gal(\gamma') -   \tau_{x'} \cup \tau_{y'} \subset  Gal^*(\gamma).$$
\end{lem}

We first prove a local version of Lemma \ref{lem:stable_global}:

\begin{lem}\label{lem:stable_local} Lemma \ref{lem:stable_global} holds under the additional assumption that the anchored decomposition of $\gamma$ has $k=0$.
 \end{lem}

\begin{proof}
Let $\Sigma_0,\Sigma_1$ be the simplices of the anchored decomposition of $\gamma$.
First notice that
if $\sigma_1$ is an edge, then $\Sigma_0,\Sigma_1$ are triangles and so by property ($*$) we have $\tau_{x'}\cup\tau_{y'}=\Sigma_0\cup \Sigma_1$, which is is convex. Thus we have $\gamma'\subset \tau_{x'}\cup \tau_{y'}$ and consequently $Gal(\gamma')  - \tau_{x'}\cup \tau_{y'} = \emptyset$ so the lemma follows. We can thus assume that $\sigma_1$ is a vertex~$v$. \\

\noindent \textbf{Case 1.} $\measuredangle_v(x', y') \geq \pi$.\\

Then~$\gamma'$ passes through~$v$. Thus $Gal(\gamma')  - \tau_{x'}\cup \tau_{y'} = \emptyset$ and the lemma follows.\\

\noindent \textbf{Case 2.} $\measuredangle_v(x', y') <\pi$.\\

Since the decomposition was anchored, Corollary~\ref{cor:bounded_below_angle} implies $ \measuredangle_v(x, y)\geq 4 \alpha$. Moreover, Definition~\ref{lem:continuity_proj} and the fact that $$d(v,x')\geq d(v,x)-d(x,x')\geq d(v,Y-U_v)-4\epsilon-\epsilon\geq 6\epsilon_0-5\epsilon\geq \epsilon_0$$ implies $\measuredangle_v(x, x')\leq \alpha$ and similarly $\measuredangle_v(y, y')\leq \alpha$. In particular, the geodesics $\pi_{v}(x)\pi_{v}(x'),\pi_{v}(y)\pi_{v}(y')$ in $lk(v)$ are disjoint and $\measuredangle_v(x, y) < \pi + 2\alpha$. Since $lk(v)$ has girth $\geq 2\pi + 4\alpha$,
the convex hull of the points $\pi_{v}(x),\pi_{v}(y),\pi_{v}(x'),\pi_{v}(y')$ in $lk(v)$ is a tree indicated in Figure~\ref{Figure_lk_tree_1} (possibly degenerate, which happens for example if $\pi_v(x)$ lies on the geodesic $\pi_v(x')\pi_v(y')$). Since $st(v)$ is convex and contains $x',y'$ by property ($*$), we have that $Gal(\gamma')$ consists of the triangles $\sigma_e$ with $e$ in the geodesic $\pi_{v}(x')\pi_{v}(y')$ in $lk(v)$. Since triangle angles are $>\alpha$, we have that among these~$\sigma_e$ only $\tau_{x'}, \tau_{y'}$ might not lie in $Gal^*(\gamma)$, as desired.
\end{proof}

\begin{figure}[H]
\begin{center}
 \scalebox{1}{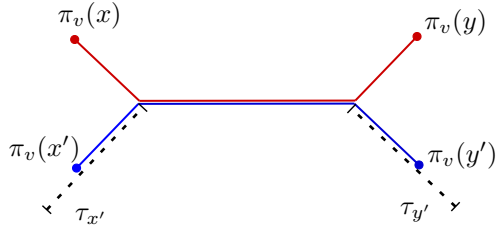}
\caption{Convex tree in the link of $v$.}
\label{Figure_lk_tree_1}
\end{center}
\end{figure}

\begin{proof}[Proof of Lemma \ref{lem:stable_global}]
We decompose $\gamma'$ as a concatenation $\gamma' = \gamma_1' \cup \cdots \cup \gamma_{k+1}'$ with $\gamma_i' = x_{i-1}' x_i'$ for a sequence of points $x_0'=x', x_1', \ldots, x_{k+1}'=y' \in \gamma'$ such that $d(x_i,x_i')<\epsilon$
for all $i$. Let $\gamma_i=x_{i-1}x_i$.
 By Lemma~\ref{lem:stable_local}, for each $1 \leq i \leq k+1$, we have  $Gal(\gamma_i') -  \tau_{x_{i-1}'}\cup \tau_{x_i'} \subset Gal^*(\gamma_i)$. To conclude that $Gal(\gamma') - \tau_{x_0'} \cup \tau_{x_{k+1}'} \subset Gal^*(\gamma)$, it suffices to prove the following.\\

\noindent \textbf{Claim.} For every $1 \leq i \leq k$, we have $ \tau_{x_i'} \subset Gal^*(\gamma_i) \cup Gal^*(\gamma_{i+1}).$ \\

To justify the claim, let us assume by contradiction that for some $1 \leq i \leq k$, the simplex $\tau_{x_i'}$ is neither contained in $Gal^*(\gamma_i)$ nor in $Gal^*(\gamma_{i+1})$. Then in particular $\tau_{x_i'} \neq \tau_{x_i}$, so by property ($*$) $\Sigma_i$ is not a triangle and thus it is an edge.
Furthermore, $\tau_{x_i}$ and $\tau_{x'_i}$ contain $\Sigma_i$, so in particular $\Sigma_i\subset Gal^*(\gamma_i)$, and thus $x'_i\notin \Sigma_i$.
For simplicity, let us denote the vertices $\sigma_i, \sigma_{i+1}$ of $\Sigma_i$ by $v_{i}$ and $v_{i+1}$ respectively. We will prove that each of $\gamma_i', \gamma_{i+1}'$ intersects $\Sigma_i$, which contradicts the convexity of $\Sigma_i$ and $x'_i\notin \Sigma_i$. To show, say, $\gamma_i'\cap \Sigma_i\neq \emptyset$, we consider the following cases.\\

\noindent \textbf{Case 1.} $\measuredangle_{v_{i}}(x'_{i-1}, x'_i) \geq \pi.$\\

Then $\gamma_i'$ passes through $v_i\in \Sigma_i$, as required.\\

\noindent \textbf{Case 2.}  $\measuredangle_{v_{i}}(x'_{i-1}, x'_i) <\pi.$\\

Since the decomposition was anchored, Corollary~\ref{cor:bounded_below_angle} implies $\measuredangle_{v_i}(x_{i-1}, x_i)\geq 4 \alpha$. Moreover, as before Definition~\ref{lem:continuity_proj} implies $\measuredangle_{v_i}(x_{i-1}, x'_{i-1})\leq \alpha$ and $\measuredangle_{v_i}(x_i, x_i')\leq \alpha$. In particular, the geodesics $\pi_{v_i}(x_{i-1})\pi_{v_i}(x_{i-1}'),\pi_{v_i}(x_i)\pi_{v_i}(x_i')$ in $lk(v_i)$ are disjoint and $\measuredangle_{v_i}(x_{i-1}, x_i) < \pi + 2\alpha$. Since $lk(v_i)$ has girth $\geq 2\pi + 4\alpha$,
the convex hull of the points $\pi_{v_i}(x_{i-1}),\pi_{v_i}(x_{i}),\pi_{v_i}(x'_{i-1}),\pi_{v_i}(x'_i)$ in $lk(v_i)$ is a tree indicated in Figure~\ref{Figure_lk_tree_3}.
Note that since $\tau_{x_i'}$ is not contained in $Gal^*(\gamma_i)$, the point $\pi_{v_i}(x_i')$ does not lie on the geodesic  $\pi_{v_i}(x_{i-1})\pi_{v_i}(x_i)$ in $lk(v_i)$.

Suppose first that the point $\pi_{v_i}(x_i)$ does not lie on the geodesic $\pi_{v_i}(x'_{i-1})\pi_{v_i}(x'_i)$ in $lk(v_i)$.
Then, since $x_i\in U_{\Sigma_i}$, the branching point indicated in the figure must be $\pi_{v_i}(v_{i+1})$ (all other branching points of $lk(v_i)$ are at distance $\geq 12\alpha-4\alpha$ from $\pi_{v_i}(x_i)$ by Definition~\ref{lem:continuity_proj}), which justifies $\gamma_i'\cap \Sigma_i\neq \emptyset$. Secondly, suppose that the point $\pi_{v_i}(x_i)$ lies on the geodesic  $\pi_{v_i}(x'_{i-1})\pi_{v_i}(x'_i)$ in $lk(v_i)$. Then, since $\pi_{v_i}(v_{i+1})$ lies on the geodesic $\pi_{v_i}(x_{i})\pi_{v_i}(x'_i)$ in $lk(v_i)$, we have $\gamma_i'\cap \Sigma_i\neq \emptyset$ as well.

\begin{figure}[H]
\begin{center}
 \scalebox{1}{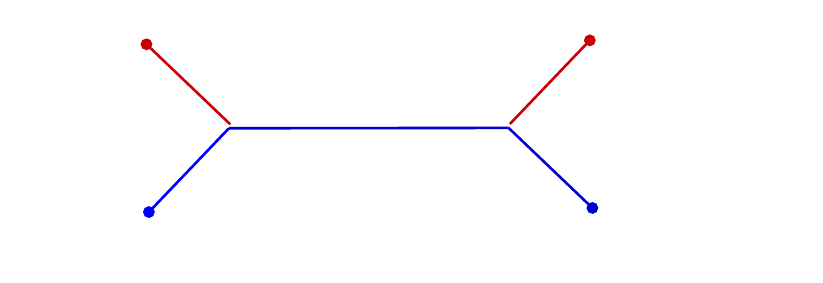}
\caption{}
\label{Figure_lk_tree_3}
\end{center}
\end{figure}

\end{proof}

\paragraph{Proof of Theorem \ref{thm:main_acyl}.} We are finally ready to prove Theorem \ref{thm:main_acyl}.
We will use the following variant of \cite[Lem~3.10]{K}.
\begin{lem}
\label{lem:fellow_travel}
For each $r\geq\epsilon>0$ there is $l>0$ satisfying the following.
Let $\gamma=xy$ be a geodesic in a $\mathrm{CAT}(-1)$ space $Y$,
and let $g$ be an isometry of~$Y$ with $d(x, gx) \leq r$ and $d(y, gy)\leq r$. For each subsegment $\gamma'$ of $\gamma$ with endpoints at distance $\geq l$ from $x$ and $y$, we have that the $r$-neighbourhood $\gamma'_+$ of $\gamma'$ in $\gamma$ satisfies $g\gamma'\subset \mathcal N_\epsilon (\gamma'_+)$.
\end{lem}

\begin{proof}[Proof of Theorem \ref{thm:main_acyl}]
Consider $r \geq\epsilon$ for $\epsilon$ as in Definition~\ref{lem:continuity_proj}, and let $L_0$, $N_0 >0$ be  such that two points of $Y$ at distance at least $L_0$ are stabilised by at most $N_0$ elements of $G$. Let $l=l(r,\epsilon)$ be as in Lemma~\ref{lem:fellow_travel}.
Since $Y$ has only finitely many isometry types of simplices, there is an upper bound $B$ on the length of a geodesic contained in the star of a vertex of $Y$ \cite[Lem~I.7.23 and Thm~I.7.28]{BH}. Let $L=L_0+4B+2l$. For a point $x \in Y$, we define the  \textit{$r$-stabiliser} of $x$ as
$$\St_r(x) = \{g \in G \mbox{ such that } d(x, gx) \leq r \}.$$
To prove acylindrical hyperbolicity, consider $x, y \in Y$ with $d(x, y)\geq L$. We will bound the size of $\St_r(x) \cap \St_r(y)$.

Let $\gamma$ be the geodesic between $x$ and $y$. By Lemma~\ref{lem:fellow_travel}, there is a subsegment $\gamma'=x'y'\subset \gamma$ of length $L_0+4B$ such that for all $g\in \St_r(x) \cap \St_r(y)$ we have $d(gx',x'_g),d(gy',y'_g)<\epsilon$ for some $x'_g,y'_g\in
\gamma'_+$. By Remark~\ref{rem:anchored}, each $x'_gy'_g$ has a subsegment with an anchored decomposition obtained by removing from $x'_gy'_g$ a subsegment of length at most $B$ at each of $x'_g,y'_g$. Consequently, by Lemma~\ref{lem:stable_global} there is a subsegment $\gamma''_g=x''_gy''_g\subset x'y'$, obtained by removing from $\gamma'$ a subsegment of length at most $B$ at each of $x',y'$,
with $g(Gal(\gamma''_g)-\tau_{x''_g}\cup\tau_{y''_g})\subset Gal^*(\gamma'_+)$. Thus the subsegment $\gamma'''$ of $\gamma'$ of length $L_0$ and centred at the midpoint of $\gamma'$ satisfies
$gGal(\gamma''') \subset Gal^*(\gamma'_+)$.
By Remark~\ref{rem:extendedsize}, there is a constant $C$ such that $Gal^*(\gamma'_+)$ contains at most $C'=(L_0+4B+2r)C$ vertices. Consequently $|\St_r(x) \cap \St_r(y)|\leq N_0C'!$, since otherwise there would be $g_0,g_1,\ldots, g_{N_0}$ with each $g_0^{-1}g_i$ fixing $Gal(\gamma''')$.
\end{proof}

\end{document}